\documentclass[reqno0,fleqn,12pt]{amsart}
\pdfoutput=1
\usepackage{dsfont}
\usepackage{bbm}
\usepackage[nodate]{datetime}
\usepackage[hmargin=25mm,top=25mm,bottom=25mm,a4paper]{geometry}
\usepackage{color}
\usepackage{subfig}
\usepackage{fancyhdr}
\usepackage{amsmath,amssymb,amsfonts,amsthm}
\usepackage{amstext}
\usepackage{amsmath}
\usepackage{amssymb}
\usepackage{enumerate}
\usepackage{amsbsy}
\usepackage{amsopn}
\usepackage{bbm,amsthm}
\usepackage{amscd}
\usepackage{amsxtra}
\usepackage{todonotes}

\newtheorem{theorem}{Theorem}[section]
\newtheorem*{theorem*}{Theorem}
\newtheorem{lemma}{Lemma}[section]

\theoremstyle{remark}

\newcommand{\CC}{\mathds{C}}
\newcommand{\abrep}{\bigg{(}}
\newcommand{\fechap}{\bigg{)}}
\newcommand{\NN}{\mathds{N}}
\newcommand{\HH}{\mathds{H}}

\newcommand{\ind}{\mathds{1}}

\DeclareMathOperator{\lcm}{lcm}

\begin{document}
\title{A note on multiplicative functions resembling the M\"obius function}
\author{Marco Aymone}
\begin{abstract}
We provide examples of multiplicative functions $f$ supported on the square free integers, such that on primes $f(p)=\pm1$ and such that
$M_f(x):=\sum_{n\leq x} f(n)=o(\sqrt{x})$. Further, by assuming the Riemann hypothesis (RH) we can go beyond $\sqrt{x}$-cancellation.
\end{abstract}

\maketitle

\section{Introduction.}
We say that $f:\NN\to\CC$ \textit{resembles} the M\"obius function $\mu$ if $f$ is multiplicative, supported on the squarefree integers, \textit{i.e.}, $f(n)=0$ whenever $n$ is divisible by some perfect square, and $f(p)\in\{+1,-1\}$ for each prime $p$.
The set of the squarefree integers $\mathcal{S}$ is an abelian group under the operation $n \circ m := \frac{\lcm(n,m)}{\gcd(n,m)}$. Further,  all the complex-valued group characters of $(\mathcal{S},\circ)$ are the real valued multiplicative functions $f$ that have support on the squarefree integers, and on primes $f(p)\in\{-1,1\}$, see \cite{Hilberdinkthegroupsquarefree}.

Let $\mathcal{P}$ be the set of prime numbers. In \cite{Maierresembling} the authors called such $f$ as a multiplicative function that resembles the M\"obius $\mu$, and their main result provide a condition on the values $(f(p))_{p\in\mathcal{P}}$ for which the partial sums $M_f(x):=\sum_{n\leq x}f(n)$ are $O(x^{1/2+o(1)})$.

If the values $(f(p))_{p\in\mathcal{P}}$ are given by independent random variables that have equal probability to be either $\pm 1$, then with probability one $M_f(x)=o(x^{1/2+\epsilon})$ for all $\epsilon>0$. Further, these partial sums are not (with probability one) $o(x^{1/2-\epsilon})$, see \cite{wintner} for these and other interesting results. Further, up to this date, the best upper bound for $M_f(x)$ can be found in \cite{tenenbaum2013} and the best $\Omega$-result can be found in \cite{harpergaussian}.

The solution of the Erd\H{o}s discrepancy problem (see \cite{taodiscrepancy}) implies that a completely multiplicative function $f:\NN\to\{-1,1\}$ has unbounded partial sums. However, a completely multiplicative function $f:\NN\to\{-1,0,1\}$ may have bounded partial sums, for instance, a real non-principal Dirichlet character $\chi$. Also, a multiplicative function $f:\NN\to\{-1,1\}$ may have bounded partial sums, see \cite{klurmancorrelation} for a complete classification of such $f$, and see \cite{klurmanchudakov} for the proof of Chudakov's conjecture. In the case $f:\NN\to\{-1,1\}$ is completely multiplicative there are known examples for which $M_f(x)$ is $O(\log x)$, see \cite{Borweindicrepancy}.

Here we address the following question:\\
\textit{For which values of $\alpha>0$ there exists a multiplicative function $f$ resembling the M\"obius function $\mu$ such that its partial sums $M_f(x)$ are $o(x^\alpha)$?}

\begin{theorem}\label{theorema 1} There exists a multiplicative function $f$ resembling $\mu$ and such that $M_f(x)=o(\sqrt{x})$. Further, if we assume RH, there exists $f$ such that
$M_f(x)=o( x^{2/5+\epsilon})$, for any $\epsilon>0$.
\end{theorem}

Further:

\begin{theorem}\label{theorem 2} Let $f$ be a multiplicative function resembling $\mu$. Let $p$ be a generic prime number. Assume that for some real non-principal Dirichlet character $\chi$ and for some constant $c>0$
\begin{equation}\label{equation strongly pretentious condition}
\sum_{p\leq x}|1-f(p)\chi(p)|\ll \frac{\sqrt{x}}{\exp(c\sqrt{\log x})}.
\end{equation}
Then for some $\lambda>0$
\begin{equation}\label{equation consequence of strongly pretentiousness}
M_f(x)\ll \frac{\sqrt{x}}{\exp(\lambda(\log x )^{1/4})}.
\end{equation}
\end{theorem}

As Theorem \ref{theorem 2} suggests, our examples of $f$ resembling $\mu$ with small partial sums are related to the real non-principal Dirichlet characters $\chi$. Indeed, the proof of Theorem \ref{theorema 1} is elementary in the following sense: We begin with a real non-principal Dirichlet character $\chi$, extend it to a completely multiplicative function $g:\NN\to\{-1,1\}$ and then we restrict it to the squarefree integers $f=\mu^2g$. The partial sums $M_f(x)$ are $o(x^{2/5+\epsilon})$ under RH, and unconditionally $\ll \frac{\sqrt{x}}{\exp(\lambda (\log x)^{1/4})}$ for some $\lambda>0$.

This raises the question of how small the partial sums $M_f(x)$ can be for $f$ resembling $\mu$ and such that $f=\mu^2 g$, where $g:\NN\to\{-1,1\}$ is a completely multiplicative extension of a real non-principal Dirichlet character $\chi\mod k$, \textit{i.e.}, $g$ is completely multiplicative, $g(n)=\chi(n)$ whenever $\gcd(n,k)=1$ and for each prime $p|k$, $g(p)=\pm1$. It is worth mentioning that the best upper bound we can obtain for $M_f(x)$ for such $f$ seems to be $o(x^{1/4})$, and further the claim $M_f(x)=o(x^{\alpha})$ for some $\alpha<1/2$ is linked with zero free regions for $\zeta$. Indeed, we have the following:
\begin{theorem}\label{theorem 3} Let $f=\mu^2g$ where $g:\NN\to\{-1,1\}$ is a completely multiplicative extension of a real non-principal Dirichlet character $\chi$. Assume that RH holds for the $L$-function $L(s,\chi)$, \textit{i.e.}, all the zeros on the half plane $\mathbb{H}_0:=\{z=\sigma+it\in\CC:\sigma>0\}$ of $L(s,\chi)$ have real part equals to $1/2$. If $M_f(x)$ is $o(x^{\alpha})$ for some $\alpha>0$, then:\\ i. $\alpha\geq 1/4$;\\ 
ii. $\zeta$ has no zeros in the half plane $\HH_{2\alpha}$.
\end{theorem}

\noindent \textbf{Acknowledgements.} I would like to thank Adam Harper for fruitful email exchanges and for suggesting that the first part of Theorem 1.1 could  be obtained by adjusting a real Dirichlet character. I was partially supported by UFMG -- Bolsa Rec\'em Doutor no. 23853.

\section{Proofs of the main results}
\subsection*{Notation} Here $M_f(x):=\sum_{n\leq x}f(n)$. We use both $f(x)\ll g(x)$ and $f(x)=O(g(x))$ whenever there exists a constant $C>0$ such that for all large $x>0$ we have that $|f(x)|\leq C|g(x)|$. Further, $\ll_\delta$ means that the implicit constant may depend on $\delta$. The standard $f(x)=o(g(x))$ means that $\lim_{x\to\infty}\frac{f(x)}{g(x)}=0$. We let $\mathcal{P}$ for the set of primes and $p$ for a generic element of $\mathcal{P}$. The notation $p^k\| n$ means that $k$ is the largest power of $p$ for which $p^k$ divides $n$. The M\"obius function is denoted by $\mu$, \textit{i.e.}, the multiplicative function with support on the square free integers and such that at the primes $\mu(p)=-1$. Dirichlet convolution is denoted by $\ast$. Given a subset $A\subset\NN$, we denote by $\ind_A(n)$ the characteristic function of $A$. Finally, $\omega(k)$ is the number of distinct primes that divide a certain $k$.
\subsection{Proof of Theorem \ref{theorem 2}} We begin with the following
\begin{lemma}\label{lemma 3} Let $h:\NN\to[0,\infty)$ be a multiplicative function such that:\\
i. $h(p)\leq2$ and $h(p^k)\leq h(p)$, for all primes $p$ and all powers $k\geq 2$;\\
ii. For some constant $c>0$, $\sum_{p\leq x}h(p)\ll \frac{\sqrt{x}}{\exp(c\sqrt{\log x})}$.\\
Then there exists a $\delta>0$ such that $M_h(x)\ll \frac{\sqrt{x}}{\exp(\delta\sqrt{\log x})}.$
\end{lemma}
\begin{proof} We are going to show that the series
\begin{equation*}\sum_{n=1}^\infty\frac{h(n)\exp(\delta\sqrt{\log n})}{\sqrt{n}}
\end{equation*}
converges for some small $0<\delta<c/2$, and hence, the proof of the desired result is obtained either by partial summation or by Kroenecker's Lemma (see \cite{shiryaev} pg. 390).

Since $\sqrt{\log n}= \sqrt{\sum_{p^k\|n} \log p^k } \leq \sum_{p^k\| n} \sqrt{ \log p^k}$ we have that
\begin{equation*}\sum_{n\leq x}\frac{h(n)\exp(\delta\sqrt{\log n})}{\sqrt{n}}\leq\sum_{n\leq x}\frac{\tilde{h}(n)}{\sqrt{n}},
\end{equation*}
where $\tilde{h}$ is the multiplicative function such that $\tilde{h}(p^k)=\exp(\delta\sqrt{\log p^k})h(p^k)$, for all primes $p$ and all powers $k$. Hence, by the Euler product formula, we only need to show that the series $\sum_{p\in\mathcal{P}}\sum_{k=1}^\infty\frac{\tilde{h}(p^k)}{p^{k/2}}$ converges.

Let $0<\delta<c/2$ be small such that $\frac{\exp(\delta \sqrt{\log p})}{\sqrt{p}}<1$ for all $p\in\mathcal{P}$. Condition i. combined with the formula for the the sum of a geometric series implies
\begin{equation}\label{equation lema 3}
\sum_{k=2}^\infty\frac{\tilde{h}(p^k)}{p^{k/2}}\leq h(p)\frac{\exp(2\delta \sqrt{\log p})}{p}\frac{1}{1-\frac{\exp(\delta \sqrt{\log p})}{\sqrt{p}}}\ll_{\delta} \frac{h(p)\exp(2\delta \sqrt{\log p})}{\sqrt{p}}.
\end{equation}

Put $T(x)=0$ for $0\leq x <1$ and $T(x)=\sum_{p\leq x}h(p)$ for $x\geq1$. We have that:
\begin{align*}
\sum_{p\leq x} \frac{h(p)\exp(2\delta\sqrt{\log p})}{\sqrt{p}}&=\int_{1}^x \frac{\exp(2\delta\sqrt{\log t})}{\sqrt{t}} dT(t)\\
&\ll T(x)\frac{\exp(2\delta\sqrt{\log x})}{\sqrt{x}}+\int_{1}^x T(t)\frac{\exp(2\delta\sqrt{\log t})}{t^{3/2}}dt\\
&\ll \frac{1}{\exp((c-2\delta)\sqrt{\log x})}+\int_{1}^x \frac{1}{t\exp((c-2\delta)\sqrt{\log t})}dt\\
&\ll 1.
\end{align*}
This estimate combined with (\ref{equation lema 3}) gives that $\sum_{p\in\mathcal{P}}\sum_{k=1}^\infty\frac{\tilde{h}(p^k)}{p^{k/2}}$ converges. \end{proof}

\begin{lemma}\label{lema 4} Let $f:\NN\to\{-1,1\}$ be completely multiplicative. Assume that for some real non-principal Dirichlet character $\chi\mod k$ $f$ satisfies (\ref{equation strongly pretentious condition}). Then for some $\delta>0$, $M_f(x)\ll \frac{\sqrt{x}}{\exp(\delta\sqrt{\log x})}$.
\end{lemma}
\begin{proof} Let $h=f\ast \chi^{-1}$, where $\chi^{-1}$ is the Dirichlet inverse of $\chi$. Then $\chi^{-1}$ is multiplicative and it is supported on the square free integers. It follows that for each prime $p$ and any power $k$:
\begin{align*}
|h(p^k)|&=|f\ast\chi^{-1}(p^k)|=|f(p^k)+f(p^{k-1})\chi^{-1}(p)|=|f(p^k)||1-f(p)\chi(p)|\\
&=|1-f(p)\chi(p)|=|h(p)|.
\end{align*}
Hence $|h|$ satisfies the conditions i-ii of Lemma \ref{lemma 3}. Since
$f=h\ast \chi$, it follows that $M_f(x)=\sum_{n\leq x}h(n) M_{\chi}(x/n)$, and since $\chi$ has (uniformly) bounded partial sums, it follows that $M_f(x)\ll_\chi M_{|h|}(x)$. \end{proof}
We complete the proof of Theorem \ref{theorem 2} with the following

\begin{lemma}\label{lemma 1} Let $g:\NN\to\{-1,1\}$ be completely multiplicative and such that
\begin{equation*}
M_g(x)\ll \frac{\sqrt{x}}{\exp(\delta \sqrt{\log x})}.
\end{equation*}
Let $f=\mu^2 g$. Then $M_f(x)$ satisfies (\ref{equation consequence of strongly pretentiousness}).
\end{lemma}
\begin{proof} Let $h:=f\ast g^{-1}$, where $g^{-1}$ is the Dirichlet inverse of $g$. Let $F$, $G$ and $H$ be the associated Dirichlet series of $f$, $g$ and $h$ respectively. The Euler product formula yields
\begin{equation*}
G(s)=\prod_{p\in\mathcal{P}}\bigg{(}1-\frac{g(p)}{p^s}\bigg{)}^{-1},\;F(s)=\prod_{p\in\mathcal{P}}\bigg{(}1+\frac{g(p)}{p^s}\bigg{)}.
\end{equation*}
Since $h=f\ast g^{-1}$:
\begin{equation*}
H(s)=\frac{F(s)}{G(s)}=\prod_{p\in\mathcal{P}}\bigg{(}1-\frac{1}{p^{2s}} \bigg{)}.
\end{equation*}
Thus, $h$ has support on the perfect squares and $h(n)=\ind_{\NN}(\sqrt{n}) \mu(\sqrt{n})$. Further, the Vinogradov-Korobov zero free region for $\zeta$ implies that $M_\mu(x)\ll x \exp(-c\sqrt{\log x})$, for some constant $c>0$. Hence
\begin{equation}\label{equation somas parcias de h}
M_h(x)=M_{\mu}(\sqrt{x})\ll \frac{\sqrt{x}}{\exp(c\sqrt{\log \sqrt{x}})}.
\end{equation}
The Dirichlet hyperbola method yields: For all $U\geq 1$ and $V\geq 1$ such that $UV=x$, we have
\begin{equation}\label{equation dirichlet hyperbola}
M_f(x)=\sum_{n\leq U}h(n)M_g\bigg{(}\frac{x}{n}\bigg{)}+\sum_{n\leq V}g(n)M_h\bigg{(}\frac{x}{n}\bigg{)}-M_g(V)M_h(U):=A+B-C.\\
\end{equation}

We choose $V=\exp(\epsilon(\sqrt{\log x}))$ where $0<\epsilon<\frac{c}{\sqrt{2}}$ and $U=\frac{x}{V}$. Further, $\lambda>0$ is a parameter
$\lambda<\min(\delta\sqrt{\epsilon},\frac{c}{\sqrt{2}}-\epsilon)$.

\noindent  \textit{Estimate for $A$.}
\begin{align*}
|A|&\leq \sum_{n\leq U} \ind_{\NN}(\sqrt{n})|M_g(x/n)|=\sum_{n\leq \sqrt{U}} |M_g(x/n^2)|\\
&\ll\sum_{n\leq \sqrt{U}} \frac{\sqrt{x}}{n}\frac{1}{\exp(\delta\sqrt{\log x/n^2})}\ll \frac{\sqrt{x}\log U}{\exp(\delta\sqrt{\log x/U})}\\
&\ll \frac{\sqrt{x}\exp(\log\log x)}{\exp(\delta\sqrt{\log V})}\ll \frac{\sqrt{x}\exp(\log\log x)}{\exp(\delta\sqrt{\epsilon}(\log x)^{1/4})}\\
&\ll \frac{\sqrt{x}}{\exp(\lambda (\log x)^{1/4})},
\end{align*}
since our $\lambda<\delta\sqrt{\epsilon}$.

\noindent  \textit{Estimate for $B$.} By (\ref{equation somas parcias de h}) we obtain:
\begin{align*}
|B|&\leq \sum_{n\leq V} |M_h(x/n)|\ll \sum_{n\leq V} \sqrt{\frac{x}{n}}\exp\bigg{(}-\frac{c}{\sqrt{2}} \sqrt{\log x/n)} \bigg{)}\\
&\ll \frac{\sqrt{x}}{\exp\bigg{(}\frac{c}{\sqrt{2}} \sqrt{\log x/V} \bigg{)}}\sum_{n\leq V} \frac{1}{\sqrt{n}}\ll \frac{\sqrt{x}}{\exp\bigg{(}\frac{c}{\sqrt{2}} \sqrt{\log x-\log V } \bigg{)}}\cdot  \sqrt{V}\\
&\ll \frac{\sqrt{x}}{\exp\bigg{(}\frac{c}{\sqrt{2}} \sqrt{\log x-\epsilon(\log x )^{1/2}} -\frac{\epsilon}{2}\sqrt{\log x} \bigg{)}}\\
&\ll \frac{\sqrt{x}}{\exp(\lambda(\log x )^{1/4})},
\end{align*}
since our $0<\lambda<\frac{c}{\sqrt{2}}-\frac{\epsilon}{2}$.

\noindent  \textit{Estimate for $C$.}
\begin{align*}
C&\ll\frac{\sqrt{V}}{\exp(\delta \sqrt{\log V})}\frac{\sqrt{U}}{\exp(\frac{c}{\sqrt{2}} \sqrt{\log U})} \ll \frac{\sqrt{UV}}{\exp(\delta \sqrt{\log V}))}\\
&\ll \frac{\sqrt{x}}{\exp(\lambda(\log x)^{1/4})}.
\end{align*}
\end{proof}

\subsection{Proof of Theorem \ref{theorema 1}}
The first part is a consequence from Theorem \ref{theorem 2} proved above. Next we are going to proof the part that depends on RH.

We say that $f:\NN\to\{-1,+1\}$ is a completely multiplicative extension of a real character $\chi\mod k$ if $f$ is completely multiplicative and $f(n)=\chi(n)$ whenever $\gcd(n,k)=1$. The following result is closely related to corollary 4 and 5 of \cite{Borweindicrepancy}:

\begin{lemma}\label{lemma 2}
Let $g:\NN\to\{-1,1\}$ be the completely multiplicative extension of a real non-principal Dirichlet character $\chi\mod k$ such that:
\begin{align*}
g(n)&= \chi(n),\mbox{ if } \gcd(n,k)=1,\\
g(p)&=1, \mbox{ for each prime }p|k.
\end{align*}
Then:
\begin{equation*}
\limsup_{x\to\infty} \frac{|M_g(x)|}{(\log x)^{\omega(k)}}\leq \frac{\max_{y\geq 1} |M_\chi(y)|}{\omega(k)!}\prod_{p|k}\frac{1}{\log p}.
\end{equation*}
\end{lemma}
\begin{proof}
Let $g$ be as above and $h=g\ast \chi^{-1}$. Let $G$, $H$ and $L$ be the Dirichlet series of $g$, $h$ and $\chi$ respectively. Observe that
\begin{equation*}
G(s)=H(s)L(s)=L(s)\prod_{p|k}\frac{1}{1-\frac{1}{p^s}}.
\end{equation*}
Let $\tilde{h}(n)=n h(n)$. Observe that $\sum_{n=1}^\infty\frac{\tilde{h}(n)}{n^{s}}=H(s-1)$ converges for all $s=\sigma+it$ with $\sigma>1$. Further, $H(s-1)$ has pole at $s=1$ of order $\omega(k)$, since
\begin{equation*}
1-\frac{1}{p^s}\sim s\log p.
\end{equation*}
In particular
\begin{equation*}
\sum_{n=1}^\infty\frac{\tilde{h}(n)}{n^s}\sim \frac{1}{(s-1)^{\omega(k)}}\prod_{p|k}\frac{1}{\log p}.
\end{equation*}
Further, $\tilde{h}(n)\geq 0$. By the Theorem of Hardy-Littlewood-Karamata (see \cite{tenenbaumlivro}, Theorem 8, pg. 227) we obtain that
\begin{equation*}
M_h(x)=\sum_{n\leq x}\frac{\tilde{h}(n)}{n}\sim \frac{1}{\omega(k)!}\prod_{p|k}\frac{\log x}{\log p}.
\end{equation*}
Since $g=h\ast\chi$, we have:
\begin{equation*}
M_g(x)=\sum_{n\leq x}h(n)M_\chi \abrep\frac{x}{n} \fechap.
\end{equation*}
Thus $|M_g(x)|\leq (\max_{y\geq 1}|M_\chi(y)|)M_h(x)$ completes the proof. \end{proof}
\begin{proof}[Proof of Theorem \ref{theorema 1}]
Let $g$ be as in Lemma \ref{lemma 2}. In particular $M_g(x)\ll x^{\alpha}$ for any $\alpha>0$. Let $f=\mu^2 g$ and $h=f\ast g^{-1}$. Thus, as in the proof of Lemma \ref{lemma 1}, $h(n)=\ind_\NN(\sqrt{n})\mu(\sqrt{n})$. Under RH, we have for any $\epsilon>0$:
\begin{equation*}
M_{h}(x)\ll x^{1/4+\epsilon}.
\end{equation*}
Next, we proceed with the Dirichlet Hyperbola method in the same line of reasoning of the proof of Lemma \ref{lemma 1}.  Let $A,B$ and $C$ be defined as in (\ref{equation dirichlet hyperbola}); $V=x^{1/5}$ and $U=x^{4/5}$. It is worth mentioning that these choices for $U$ and $V$ are optimal.

\noindent  \textit{Estimate for $A$.}
\begin{align*}
A\ll x^\alpha U^{\frac{1}{2}-\alpha}\ll x^{\alpha}x^{4/5(1/2-\alpha)}
\ll x^{2/5+\alpha/5}.
\end{align*}

\noindent  \textit{Estimate for $B$.}
\begin{align*}
B\ll \sum_{n\leq V} \frac{x^{1/4+\epsilon}}{n^{1/4+\epsilon}}
\ll x^{1/4+\epsilon}V^{3/4-\epsilon}
\ll x^{1/4+\epsilon}x^{1/5(3/4-\epsilon)}
\ll x^{2/5+4/5\epsilon}.
\end{align*}
\noindent  \textit{Estimate for $C$.}
\begin{align*}
C\ll V^\alpha U^{1/4+\epsilon}= x^{\alpha/5 +4/5(1/4+\epsilon)}=x^{1/5+\alpha/5+4\epsilon/5}
\end{align*}
We complete the proof by choosing $\alpha>0$ and $\epsilon>0$ arbitrarily small. \end{proof}

\subsection{Proof of Theorem \ref{theorem 3}}
\begin{proof} Let $\chi$ be a real non-principal Dirichlet character and $L(s,\chi)$ be its Dirichlet series. Assume that RH holds for $L(s,\chi)$. Let $g:\NN\to\{-1,1\}$ be a completely multiplicative extension of $\chi$ and $f=\mu^2 g$. Let $F(s)$ and $G(s)$ be the Dirichlet series of $f$ and $g$ respectively. It follows that $G(s)$ is analytic in the half plane $\HH_0$ and share same zeros (with same multiplicty) with $L(s,\chi)$. On the other hand, the hypothesis $M_f(x)=o(x^\alpha)$ implies that $F$ is analytic in $\HH_\alpha$. Observe that $\frac{F(s)}{G(s)}=\frac{1}{\zeta(2s)}$. Since $\frac{1}{\zeta(2s)}$ is analytic in some open set containing the closed half plane $\HH_{1/2}^*$ and has a zero only at $s=1/2$, it follows that $F$ has the same zeros as $G$ (with the same multiplicity, with a possible exception at $s=1/2$) in this open set. Hence the zeros of $\zeta(2s)$ are poles for $F(s)$, which implies that $\alpha\geq 1/4$. Further, the RH for $L(s,\chi)$ implies that $\frac{F(s)}{G(s)}$ is analytic where $F$ is; In particular it is analytic in $\HH_\alpha$. It follows that $\frac{1}{\zeta(2s)}$ is analytic in $\HH_\alpha$. \end{proof}

{\small{\sc \noindent Marco Aymone \\
Departamento de Matem\'atica, Universidade Federal de Minas Gerais, Av. Ant\^onio Carlos, 6627, CEP 31270-901, Belo Horizonte, MG, Brazil.} \\
\textit{Email address:} marco@mat.ufmg.br}
\vspace{0.5cm}

\end{document}